\documentclass[graybox]{svmult}

\usepackage{helvet}         
\usepackage{courier}        
\usepackage{type1cm}        
\usepackage{makeidx}         
\usepackage{graphicx}        
\usepackage{multicol}        
\usepackage[bottom]{footmisc}
\usepackage{bm}
\usepackage{mathtools}

\makeindex             

\usepackage{latexsym}
\usepackage{dsfont}
\usepackage{bbm}
\usepackage{amssymb}
\usepackage{amsmath}

\newtheorem{Theorem}{Theorem}[section]
\newtheorem{Lemma}[Theorem]{Lemma}

\newtheorem{Prop}[Theorem]{Proposition}
\newtheorem{Rem}[Theorem]{Remark}
\newtheorem{Exa}[Theorem]{Example}

\def\cL{\mathcal{L}}

\def\cZ{\mathcal{Z}}

\def\bbD{\mathbb{D}} 
\def\Erw{\mathbb{E}}

\def\L{\mathbb{L}}

\def\N{\mathbb{N}}
  
\def\Prob{\mathbb{P}} 

\def\R{\mathbb{R}}

\def\bfE{\text{\rm\bfseries E}}
\def\bfe{\text{\rm\bfseries e}}

\def\bfP{\text{\rm\bfseries P}}

\def\eps{\varepsilon}
\def\vph{\varphi}
\def\vth{\vartheta}

\def\1{\vec{1}}
\def\3{{\ss}}

\def\eqdist{\stackrel{d}{=}}

\def\idist{\stackrel{d}{\to}}

\def\weakly{\stackrel{w}{\to}}
\def\RA{\Rightarrow}

\def\wh{\widehat}

\def\geq{\geqslant}

\def\IRg{\R_{\scriptscriptstyle >}}
\def\IRge{\R_{\scriptscriptstyle\geqslant}}
\def\IRl{\R_{\scriptscriptstyle <}}

\def\Geom{\textit{Geom}\hspace{1pt}}

\allowdisplaybreaks[4] 

\begin{document}

\title*{Linear fractional Galton-Watson processes in random environment and perpetuities}
\titlerunning{Linear fractional GWPRE and perpetuities}
\author{Gerold Alsmeyer$^{1}$} 
\institute{$^{1}$ Inst.~Math.~Stochastics, Department
of Mathematics and Computer Science, University of M\"unster,
Orl\'eans-Ring 10, D-48149 M\"unster, Germany.\at
\email{gerolda@math.uni-muenster.de}\at
Work partially funded by the Deutsche Forschungsgemeinschaft (DFG) under Germany's Excellence Strategy EXC 2044--390685587, Mathematics M\"unster: Dynamics--Geometry--Structure.}

\maketitle

\abstract{Linear fractional Galton-Watson branching processes in i.i.d.~random environment are, on the quenched level, intimately connected to random difference equations by the evolution of the random parameters of their linear fractional marginals. On the other hand, any random difference equation defines an autoregressive Markov chain (a random affine recursion) which can be positive recurrent, null recurrent and transient and which, as the forward iterations of an iterated function system, has an a.s.~convergent counterpart in the positive recurrent case given by the corresponding backward iterations. The present expository article aims to provide an explicit view at how these aspects of random difference equations and their stationary limits, called perpetuities, enter into the results and the analysis, especially  in quenched regime. Although most of the results presented here are known, we hope that the offered perspective will be welcomed by some readers.}

\bigskip



\section{Introduction}

Let us begin with a disclaimer. This work about Galton-Watson branching processes with linear fractional offspring distributions in i.i.d.~random environment is neither a research paper nor a survey. It is rather meant as an attempt to provide, via a collection of selected results, some (hopefully) new vantage points of the interesting and rather explicit connections between this class of branching processes when studied in quenched regime and random difference equations and their stationary limits, called perpetuities. The latter have attracted a lot of interest in the last two decades, not at least due to their appearance in various other fields of probability theory. We refer to the recent monograph by Buraczewski, Damek and Mikosch \cite{BurDamMik:16} for further information and literature. Another recent monograph by Kersting and Vatutin \cite{KerstingVatutin:17} provides an excellent account of the current state-of-the-art of branching processes in random environment and is here especially recommended in places where this text remains terse on accounting for relevant references. To keep the presentation at reasonable length while stressing our particular perspective, we restrict ourselves to a collection of results of moderate technical level, and mostly in quenched regime, that nicely illustrate the interplay between linear fractional branching in random environment and random difference equations. Owing to the same constraint, only the subcritical case is considered at greater length, while keeping the sections on supercritical and critical processes relatively short which we deem sufficient for our purposes.

\vspace{.1cm}
As a motivation, consider the classical Galton-Watson branching process (GWP) $(Z_{n})_{n\ge 0}$ with $Z_{0}=1$ and offspring generating function (g.f.)
$$ f(s)\ =\ \sum_{n\ge 0}p_{n}s^{n} $$ 
for $s\in [0,1]$. Then $f^{n}$, the $n$-fold iteration of $f$, equals the g.f. of $Z_{n}$ for each $n\in\N_{0}$, where $\N_{0}=\{0,1,2,\ldots\}$ and $f^{0}(s):=s$. Among the very few examples that allow to compute all $f^{n}$ in closed form from $f$, the presumably most prominent one is the linear fractional case when
\begin{equation}\label{eq:def LF gf}
\frac{1}{1-f(s)}\ =\ \frac{a}{1-s}+b
\end{equation}
for parameters $a,b\in\IRg=(0,\infty)$, $a+b\ge 1$, and $s\in [0,1)$. Let us write $LF(a,b)$ for the associated distribution on $\N_{0}$ which, as can be seen from \eqref{eq2:def LF gf}, is a mixture of $\delta_{0}$, the point mass at $0$, and a geometric distribution on the positive integers $\N$ which has parameter $p$ and is denoted $\Geom_{+}(p)$. It is in fact a pure geometric law with parameter $a$ iff $a+b=1$, so
$$ LF(a,1-a)\ =\ \Geom_{+}(a). $$
Simple computations show that
\begin{equation}\label{eq2:def LF gf}
f(s)\ =\ p_{0}\,+\,(1-p_{0})\frac{ps}{1-(1-p)s}
\end{equation}
with $(p_{0},p)$ determined by the equations
\begin{align}\label{eq:(a,b)<->(p_0,p)}
a\ =\ \frac{p}{1-p_{0}}\quad\text{and}\quad b\ =\ \frac{1-p}{1-p_{0}}.
\end{align}
Conversely, $p=a(a+b)^{-1}$ and $p_{0}=(a+b-1)(a+b)^{-1}$. {\color{black} Therefore, by what has been stated above,
\begin{equation}\label{eq:LF->Geomplus}
LF(a,b)\ =\ \frac{a+b-1}{a+b}\,\delta_{0}\ +\ \frac{1}{a+b}\,\Geom_{+}\left(\frac{a}{a+b}\right).
\end{equation}
}
We further note that the offspring mean $m$, say, satisfies $m:=f'(1)=\sum_{n\ge 1}np_{n}=a^{-1}$, that $f''(1)=2a^{-2}b$, and that the extinction probability equals $q=b^{-1}(a+b-1)$ in the supercritical case $m>1$. For further details, the reader may consult the classical monograph by Athreya and Ney \cite[Sect.~A.4]{Athreya+Ney:72}. Finally, we point out for later use that, if $a+b=1$ and thus $f$ belongs to the geometric law $\Geom_{+}(a)$ on $\N$, then, for any $\gamma\in (0,1]$, the g.f.
\begin{equation}\label{eq:geom law transformation}
\frac{f(\gamma s)}{f(\gamma)}\ =\ \frac{(1-(1-a)\gamma)s}{1-(1-a)\gamma s},\quad s\in [0,1]
\end{equation}
belongs to the geometric law $\Geom_{+}(1-(1-a)\gamma)$.

\vspace{.2cm}
Putting $\vph(s):=(1-s)^{-1}$ for $s\in [0,1)$, which is a bijection from $[0,1)$ to $[1,\infty)$, and $g(s):=as+b$ for $s\in\R$, Eq.~\eqref{eq:def LF gf} may be restated as 
\begin{gather*}
\vph\circ f(s)\ =\ g\circ\vph(s)\\
\shortintertext{or, equivalently,}
f(s)\ =\ \vph^{-1}\circ g\circ\vph(s)
\end{gather*}
for $s\in [0,1)$. Therefore the iterations of $f$, viewed as a dynamical system on $[0,1)$, coincide up to conjugation with respect to $\vph$, with the iterations of $g$, viewed as a dynamical system on $[1,\infty)$, giving
\begin{equation*}
f^{n}(s)\ =\ \vph^{-1}\circ g^{n}\circ\vph(s)
\end{equation*}
for each $n\in\N_{0}$. Using $g^{n}(s)=a^{n}s+b(a^{n-1}+...+a+1)$, we thus find
\begin{equation*}
\frac{1}{1-f^{n}(s)}\ =\ \frac{a^{n}}{1-s}\,+\,b(a^{n-1}+\ldots+a+1).
\end{equation*}
and therefore $\cL(Z_{n})=LF(a^{n},b(a^{n-1}+\ldots+1))$, where $\cL(X)$ means the law of $X$. This conjugation argument or, equivalently, the use of \eqref{eq:def LF gf} rather than  \eqref{eq2:def LF gf} to find the iterations of $f$ and thus the laws of all $Z_{n}$ is easier than the approach described in \cite{Athreya+Ney:72}. It may also be found in \cite[p.\,3ff]{KerstingVatutin:17}.

\vspace{.4cm}
The last observation becomes even more striking in the situation when the offspring laws are still linear fractional but varying with respect to an i.i.d.~random environment, thus leading to a Galton-Watson process in random environment (GWPRE). More precisely, let $\bfe:=(A_{n},B_{n})_{n\ge 1}$ be a sequence of i.i.d.~random vectors (the environment) with generic copy $(A,B)$ such that
\begin{equation}\label{eq:parameter settings}
\Prob(A>0,\,B>0,\,A+B\ge 1)\,=\,1.
\end{equation}
Put $\bfe_{n}:=(A_{n},B_{n})$, $\bfe_{1:n}:=(\bfe_{1},\ldots,\bfe_{n})$ and $\bfe_{\geq n}:=(\bfe_{n},\bfe_{n+1},\ldots)$, so $\bfe=\bfe_{\geq 1}$.
Define the \emph{random g.f.} $f_{n}=f(\bfe_{n},\cdot)$ by
$$ \frac{1}{1-f_{n}(s)}\ =\ \frac{A_{n}}{1-s}\,+\,B_{n} $$
for $n\in\N$. The pertinent random linear fractional distribution is denoted by $(P_{n,k})_{k\ge 0}$, with generic copy $(P_{k})_{k\ge 0}=LF(A,B)$. Suppose that, conditioned upon $\bfe_{1:n}$, the members of the $(n-1)^{th}$ generation produce offspring in accordance with the linear fractional distribution $(P_{n,k})_{k\ge 0}=LF(A_{n},B_{n})$ having g.f. $f_{n}$. Then 
$$ f_{1:n}\ :=\ f_{1}\circ\ldots\circ f_{n} $$
equals the g.f. of (the quenched law of) $Z_{n}$ given $\bfe_{1:n}$, and also given $\bfe$. It satisfies
\begin{equation}\label{eq:basic identity for f_1:n}
\vph\circ f_{1:n}(s)\ =\ \frac{1}{1-f_{1:n}(s)}\ =\ \frac{\Pi_{n}}{1-s}\,+\,R_{n}\ =\ g_{1:n}\circ\vph(s)
\end{equation}
for $s\in [0,1)$, where 
$$\ g_{n}(x)=g(\bfe_{n},x):=A_{n}x+B_{n},\quad\Pi_{n}:=\prod_{k=1}^{n}A_{k}\quad\text{and}\quad R_{n}:=\sum_{k=1}^{n}\Pi_{k-1}B_{k} $$ 
for $n\in\N$. We thus see that all quenched laws are linear fractional, namely
\begin{gather}\label{eq:quenched law Z_n}
\cL(Z_{n}|\bfe_{1:n})\ =\ \cL(Z_{n}|\bfe)\ =\ LF(\Pi_{n},R_{n})
\intertext{with g.f. \color{black}(see also \eqref{eq:LF->Geomplus})}
f_{1:n}(s) =\ \frac{\Pi_{n}+R_{n}-1}{\Pi_{n}+R_{n}}\,+\,\frac{1}{\Pi_{n}+R_{n}}\cdot\frac{\Pi_{n}s}{\Pi_{n}+R_{n}(1-s)}\label{eq:quenched gf Z_n}
\end{gather}
for each $n\in\N$, and that, up to conjugation, $(f_{1:n})_{n\ge 0}$ equals the sequence of backward iterations of the i.i.d.~affine linear random maps $g_{1},g_{2},\ldots$. The corresponding forward iterations $g_{n:1}(x):=g_{n}\circ\ldots\circ g_{1}(x)$, called \emph{iterated function system (IFS)}, form a Markov chain on $[1,\infty)$ with initial state $x$. This chain has been extensively studied in the literature, see e.g.~\cite{Kesten:73,Vervaat:79,GolMal:00,AlsIksRoe:09} and especially the recent monographs \cite{BurDamMik:16,Iksanov:16}, and we will review some of its essential properties in the next section. In view of these observations it appears to be natural to study properties of the linear fractional GWPRE as just introduced by drawing on results about iterations of the $g_{n}$. As an immediate consequence of \eqref{eq:basic identity for f_1:n}, we have that
\begin{equation}\label{eq:survival probab quenched}
q_{n}(\bfe_{1:n})\ :=\ \Prob(Z_{n}=0|\bfe_{1:n})\ =\ f_{1:n}(0)\ =\ 1-\frac{1}{\Pi_{n}+R_{n}}\quad\text{a.s.}
\end{equation}
for all $n\ge 1$, and we let
$$ q(\bfe)\ :=\ \lim_{n\to\infty}\Prob(Z_{n}=0|\bfe_{1:n}) $$
denote the quenched extinction probability of $(Z_{n})_{n\ge 0}$ given $\bfe$.

\vspace{.2cm}
\emph{Reversing the environment}. The trivial fact that $\bfe_{1:n}\eqdist\bfe_{n:1}$ for each $n\in\N$ allows us to study the given GWPRE up to any time $n$ under the time-reversed environment $\bfe_{n:1}$ without changing its (annealed) law. On the other hand, the quenched laws are naturally different, but the fact that they have the same distribution (as random measures, see \eqref{eq:equal quenched laws} below) will facilitate assertions about quenched asymptotic behavior that are more tangible than those without time-reversal. Let us define
\begin{gather}
\bfP\,:=\,\Prob(\cdot|\bfe),\quad\bfP^{(1:n)}\,:=\,\Prob(\cdot|\bfe_{1:n})\quad\text{and}\quad\bfP^{(n:1)}\,:=\,\Prob(\cdot|\bfe_{n:1})\nonumber
\intertext{with corresponding expectations $\bfE,\,\bfE^{(1:n)}$ and $\bfE^{(n:1)}$. Note that}
\bfP^{(1:n)}((Z_{0},\ldots,Z_{n})\in\cdot)\ \eqdist\ \bfP^{(n:1)}((Z_{0},\ldots,Z_{n})\in\cdot)\label{eq:equal quenched laws}
\end{gather}
for each $n\in\N$, in particular,
\begin{gather}
\bfP^{(n:1)}(Z_{n}\in\cdot)\ =\ LF\left(\Pi_{n},\Pi_{n}\sum_{k=1}^{n}\Pi_{k}^{-1}B_{k}\right)
\label{eq:quenched law Z_n bw}
\shortintertext{and}
q_{n}(\bfe_{n:1})\ :=\ \Prob(Z_{n}=0|\bfe_{n:1})\ =\ 1-\frac{1}{\Pi_{n}(1+\sum_{k=1}^{n}\Pi_{k}^{-1}B_{k})}\quad\text{a.s.}\label{eq:survival probab quenched bw}
\end{gather}
as one can easily check (see \eqref{eq:quenched law Z_n} and \eqref{eq:survival probab quenched}).

\vspace{.2cm}
For GWP's in i.i.d.~random environment, the usual distinction between subcritical, critical and supercritical case is based on the asymptotic behavior of the logarithm of the quenched mean $\log\bfE Z_{n}$. Provided this quantity is a.s.~finite for all $n$, it constitutes an ordinary random walk, denoted $(S_{n})_{n\ge 0}$ throughout, with generic increment $\log f'(1)$ which in the present situation equals $-\log A$.
In fact, if $Z_{0}=1$, then
$$ \log\bfE Z_{n}\ =\ -\log\Pi_{n}\ =\ -\sum_{k=1}^{n}\log A_{k}\ =:\ S_{n}\quad\text{a.s.} $$
for all $n\ge 0$. Depending on the fluctuation-type of this walk, namely
\begin{description}[(b)]\itemsep4pt
\item[] positive divergence $(S_{n}\to\infty\text{ a.s.})$,
\item[] negative divergence $(S_{n}\to-\infty\text{ a.s.})$,
\item[] oscillation $\displaystyle\left(\limsup_{n\to\infty}S_{n}=+\infty\text{ and }\liminf_{n\to\infty}S_{n}=-\infty\text{ a.s.}\right)$,
\item[] trivial $(S_{n}=0\text{ a.s~for all }n\ge 0)$,
\end{description}
the process $(Z_{n})_{n\ge 0}$ is called subcritical, supercritical, critical or strongly critical, respectively \cite[Def.~2.3]{KerstingVatutin:17}. If $\Erw\log A$ exists, this means that
\begin{align*}
(Z_{n})_{n\ge 0}\text{ is }
\begin{cases}
\hfill\text{subcritical}&\text{if }\Erw\log A>0,\\
\hfill\text{critical}&\text{if }\Erw\log A=0\text{ and }\Prob(A\ne 1)>0,\\
\text{strongly critical}&\text{if }A=1\text{ a.s.},\\
\hfill\text{supercritical}&\text{if }\Erw\log A<0.
\end{cases}
\end{align*}
Due to the very explicit knowledge of the $f_{1:n}$ in the present situation, a precise description of when each of these cases occurs is possible and in fact provided at the end of the next section after the collection of some relevant facts about iterations of random affine linear functions. Based on this, our definition of subcritical and supercritical processes will be slightly more restrictive as above and thus entail a wider definition of critical processes.

\section{Basic results for iterations of random affine linear functions}

Let us collect some essential facts about IFS generated by affine linear random functions $g_{n}(x)=A_{n}x+B_{n}$ with i.i.d.~positive random coefficients $A_{n},B_{n}$. As already stated, the \emph{forward iterations}
$$ g_{n:1}(x)\ =\ \Pi_{n}x\,+\,\sum_{k=1}^{n}\frac{\Pi_{n}}{\Pi_{k}}B_{k},\quad n\ge 0 $$
form a Markov chain with initial state $x$. The corresponding \emph{backward iterations}
$$ g_{1:n}(x)\ =\ \Pi_{n}(x)\,+\,\sum_{k=1}^{n}\Pi_{k-1}B_{k}\ =\ \Pi_{n}x\,+\,R_{n},\quad n\ge 0, $$
though having the same distribution $(g_{n:1}(x)\eqdist g_{1:n}(x))$, exhibit a very different behavior and are \emph{not} Markovian. They are in fact strictly increasing, for $A,B$ are positive. Notice that, if $g_{n:1}(x)=G_{n}(A_{1},B_{1},\ldots,$ $A_{n},B_{n})$ for a suitable function $G_{n}$, then $g_{1:n}(x)=G_{n}(A_{n},B_{n},\ldots,A_{1},B_{1})$. Goldie and Maller \cite[Thm.~2.1]{GolMal:00} have provided necessary and sufficient conditions for the stability (positive recurrence) of $(g_{n:1}(x))_{n\ge 0}$, here summarized in the subsequent proposition for the case of positive $A,B$.

\begin{Prop}\label{prop:GoldieMaller}
Suppose that $A,B$ are a.s. positive. Then the following assertions are equivalent:
\begin{description}[(b)]
\item[(a)] If $J^{-}(x):=\int_{0}^{x}\Prob(-\log A>y)\,dy=\Erw(x\wedge\log^{-}A)$ for $x>0$, then
\begin{equation}\label{eq:GolMal cond}
\Pi_{n}\,\to\,0\text{ a.s.}\quad\text{and}\quad
I_{-}\,:=\,\int_{[1,\infty)}\frac{\log x}{J^{-}(x)}\ \Prob(\log B\in dx)\ <\ \infty.
\end{equation}
\item[(b)] The so-called perpetuity
$$ R_{\infty}\ :=\ \sum_{k\ge 1}\Pi_{k-1}B_{k} $$
is a.s.~finite and, for each $x\in\R$, the backward iterations $g_{1:n}(x)$ converge a.s. (monotonically) to $R_{\infty}$, while the forward iterations $g_{n:1}(x)$ converge in distribution to $R_{\infty}$.
\end{description}
Conversely, if
\begin{equation}\label{eq:nondegeneracy}
\Prob(Ax+B=x)\,<\,1\quad\text{for all }x\in\R
\end{equation}
and at least one of the conditions in \eqref{eq:GolMal cond} fails to hold, then $R_{\infty}=\infty$ a.s.
\end{Prop}

Plainly, the law of $R_{\infty}$, if a.s. finite, forms the unique stationary distribution of the Markov chain $(g_{n:1}(x))_{n\ge 0}$. Let us further point out that for $(A,B)$ with $A+B\ge 1$ a.s., which is a necessary requirement in the above branching framework, we further have
\begin{equation}\label{eq:Sinfty>=1}
R_{\infty}\,\ge\, 1\quad\text{a.s.}
\end{equation}
Namely, if $R_{\infty}<\infty$ a.s. and thus condition \eqref{eq:GolMal cond} is valid, then $B\ge 1-A$ a.s. implies
\begin{align*}
R_{\infty}\ &\ge\ \lim_{n\to\infty}\sum_{k=1}^{n}\Pi_{k-1}(1-A_{k})\\
&=\ \lim_{n\to\infty}\sum_{k=1}^{n}(\Pi_{k-1}-\Pi_{k})\ =\ 1-\lim_{n\to\infty}\Pi_{n}\ =\ 1.
\end{align*}
This actually even shows that $R_{\infty}>1$ a.s. unless $A+B=1$ a.s.

\vspace{.2cm}
As a direct consequence of the previous proposition, we can state the following duality result for the case when $\Pi_{n}\to\infty$ a.s.

\begin{Prop}\label{prop:duality lemma}
Defining $g_{n}^{(-1)}(x):=A_{n}^{-1}x+A_{n}^{-1}B_{n}$ for $n\in\N$ (which is not the inverse of $g_{n}$), the duality relation
\begin{equation}
\frac{R_{n}}{\Pi_{n}}\ =\ \frac{g_{1:n}(0)}{\Pi_{n}}\ =\ g_{n:1}^{(-1)}(0)\ \eqdist\ g_{1:n}^{(-1)}(0)\ =\ \sum_{k=1}^{n}\Pi_{k}^{-1}B_{k}\ =:\ R_{n}^{(-1)}
\end{equation}
holds for all $n\in\N$. Moreover, if
\begin{equation}\label{eq:GolMal cond dual}
\Pi_{n}\,\to\,\infty\text{ a.s.}\quad\text{and}\quad
I_{+}\,:=\,\int_{[1,\infty)}\frac{\log x}{J^{+}(x)}\ \Prob(\log(B/A)\in dx)\ <\ \infty,
\end{equation}
where $J^{+}(x):=\Erw(x\wedge\log^{+}A)$, then $R_{\infty}^{(-1)}:=\sum_{k\ge 1}\Pi_{k}^{-1}B_{k}<\infty$ a.s. and
\begin{equation*}
\frac{R_{n}}{\Pi_{n}}\ \idist\ R_{\infty}^{(-1)}.
\end{equation*}
On the other hand, if \eqref{eq:nondegeneracy} is valid and at least one of the conditions in \eqref{eq:GolMal cond dual} fails to hold, then $R_{\infty}^{(-1)}=\infty$ a.s.
\end{Prop}

Our standing assumption \eqref{eq:parameter settings} ensures that $R_{\infty}$ and $R_{\infty}^{(-1)}$ \emph{always} exist as the strictly increasing limits of $R_{n}$ and $R_{n}^{(-1)}$, respectively. One can also easily verify that these random variables cannot be a.s. finite at the same time. Therefore, the trichotomy
\begin{description}[(C3)]\itemsep=2pt
\item[(C1)] $R_{\infty}<\infty=R_{\infty}^{(-1)}$ a.s.
\item[(C2)] $R_{\infty}^{(-1)}<\infty=R_{\infty}$ a.s.
\item[(C3)] $R_{\infty}=R_{\infty}^{(-1)}=\infty$ a.s.
\end{description}
holds, and with the help of the previous two propositions characterization of the three cases in terms of $(A,B)$ is easily provided and leads to the announced classification of the four criticality regimes of $(Z_{n})_{n\ge 0}$ that differs slightly from the one based only on the fluctuation type of the random walk $(S_{n})_{n\ge 0}$ used in \cite{KerstingVatutin:17}.

\begin{Prop}\label{prop:trichotomy}
Assuming $A,B>0$, we have that\\[1.5mm]
(C1) ocurs iff one of the following two sets of  conditions holds:
\begin{description}[~(C3.6)~]\itemsep3pt
\item[~(C1.1)] Cond.~\eqref{eq:nondegeneracy}, $\Pi_{n}\to 0$ a.s. and $I_{-}<\infty$.
\item[~(C1.2)] $B=x(1-A)$ a.s. for some $x>0$ $(\RA\Pi_{n}\to 0$ a.s.$)$.
\end{description}
(C2) occurs iff one of the following two sets of conditions holds:
\begin{description}[~(C3.6)~]\itemsep3pt
\item[~(C2.1)] Cond.~\eqref{eq:nondegeneracy}, $\Pi_{n}\to\infty$ a.s. and $I_{+}<\infty$.
\item[~(C2.2)] $B=x(A-1)$ a.s. for some $x>0$ $(\RA\Pi_{n}\to\infty$ a.s.$)$.
\end{description}
(C3) occurs iff one of the following four sets of conditions holds:
\begin{description}[~(C3.6)~]\itemsep3pt
\item[~(C3.1)] Cond.~\eqref{eq:nondegeneracy}, $\Pi_{n}\to 0$ a.s. and $I_{-}=\infty$.
\item[~(C3.2)] Cond.~\eqref{eq:nondegeneracy}, $\Pi_{n}\to\infty$ a.s. and $I_{+}=\infty$.
\item[~(C3.3)] $\liminf_{n\to\infty}\Pi_{n}=0$ and $\limsup_{n\to\infty}\Pi_{n}=\infty$ a.s.
\item[~(C3.4)] $A=1$ and thus $\Pi_{n}=1$ for all $n\in\N_{0}$ a.s.
\end{description}
\end{Prop}

Justified by the results that will be presented in the next three sections, the process $(Z_{n})_{n\ge 0}$ is called
\begin{description}\itemsep2pt
\item[] supercritical under (C1);
\item[] subcritical under (C2);
\item[] critical under any of (C3.1), (C3.2), or (C3.3);
\item[] strongly critical under (C3.4).
\end{description}

\begin{proof}
First we consider the situation when $Ax+B=x$ a.s.~for some $x\in\R$ which must be nonzero because otherwise $B=0$ a.s.~would follow. If $x>0$, then $B=x(1-A)>0$ a.s.~entails $0<A<1$ a.s. and thus $\Pi_{n}\to 0$ a.s. Now
\begin{align*}
&R_{\infty}\ =\ \lim_{n\to\infty}R_{n}\ =\ \lim_{n\to\infty}x\sum_{k=1}^{n}\Pi_{k-1}(1-A_{k})\ =\ \lim_{n\to\infty}x(1-\Pi_{n})\ =\ x
\shortintertext{together with}
&\hspace{.8cm}R_{\infty}^{(-1)}\ =\ \lim_{n\to\infty}x\sum_{k=1}^{n}\Pi_{k}^{-1}(1-A_{k})\ =\ \lim_{n\to\infty}x(\Pi_{n}^{-1}-1)\ =\ \infty
\end{align*}
shows that (C1) holds true. If $Ax+B=x$ a.s. for some $x<0$, then (C2) with
\begin{align*}
R_{\infty}^{(-1)}\ =\ \lim_{n\to\infty}x(\Pi_{n}^{-1}-1)\ =\ |x|
\end{align*}
follows in the same manner.

\vspace{.2cm}
For the remainder of the proof suppose that \eqref{eq:nondegeneracy} is valid. Then Prop. \ref{prop:GoldieMaller} ensures that (C1) holds iff $\Pi_{n}\to 0$ a.s. and $I_{-}<\infty$ (see \eqref{eq:GolMal cond}), while  Prop. \ref{prop:duality lemma} shows equivalence of (C2) with $\Pi_{n}\to\infty$ and $I_{+}<\infty$ (see \eqref{eq:GolMal cond dual}). As a consequence, (C3) must be valid in any of the remaining cases, stated as (C3.1)--(C3.4).\qed
\end{proof}

\begin{Rem}\label{rem:A+B>=1}\rm
Since $A+B\ge 1$ in the branching framework, as pointed out earlier, $B=x(1-A)$ a.s. for some $x>0$ can actually only occur if $x\ge 1$.
\end{Rem}

\begin{Rem}\label{rem:A=1}\rm
Note that $A=1$ a.s. entails $R_{n}=R_{n}^{(-1)}=\sum_{k=1}^{n}B_{k}$ for all $n\ge 1$. In other words, $(R_{n})_{n\ge 0}$ and $(R_{n}^{(-1)})_{n\ge 0}$ coincide and constitute an ordinary random walk with generic increment $B$. It is nontrivial by our standing assumption $\Prob(B>0)=1$, see \eqref{eq:parameter settings}.
\end{Rem}

\section{The subcritical case}

Let $(Z_{n})_{n\ge 0}$ be subcritical, thus $R_{\infty}^{(-1)}<\infty=R_{\infty}$ and $\Pi_{n}\to\infty$ a.s. In order to determine the quasistationary behavior of $Z_{n}$ given $Z_{n}>0$ as $n\to\infty$, put $$ h_{n}(s)\ :=\ \bfE^{(1:n)}\left(s^{Z_{n}}|Z_{n}>0\right) $$
for $n\in\N$. Then
\begin{align*}
h_{n}(s)\ =\ \frac{f_{1:n}(s)-f_{1:n}(0)}{1-f_{1:n}(0)}\ =\ 1\,-\,\frac{1-f_{1:n}(s)}{1-f_{1:n}(0)}
\end{align*}
and therefore
\begin{align*}
\frac{1}{1-h_{n}(s)}\ &=\ \frac{1-f_{1:n}(0)}{1-f_{1:n}(s)}\ =\ \frac{\Pi_{n}(1-s)^{-1}+R_{n}}{\Pi_{n}+R_{n}}\\
&=\ \frac{\Pi_{n}}{\Pi_{n}+R_{n}}\cdot\frac{1}{1-s}\,+\,\frac{R_{n}}{\Pi_{n}+R_{n}}\\
&=\ \frac{1}{1+R_{n}/\Pi_{n}}\cdot\frac{1}{1-s}\,+\,\frac{R_{n}/\Pi_{n}}{1+R_{n}/\Pi_{n}}
\end{align*}
for each $n\in\N$. In other words, 
\begin{equation}\label{eq:dist Zn given Zn>0 fw}
\bfP^{(1:n)}(Z_{n}\in\cdot|Z_{n}>0)\ =\ \Geom_{+}\left(\frac{1}{1+R_{n}/\Pi_{n}}\right)
\end{equation}
is a.s.~geometric on $\N$ and thus again linear fractional. However, it fluctuates in accordance with $R_{n}/\Pi_{n}$ which converges only in distribution. The same observation is made for the pertinent quenched survival probability (see also \eqref{eq:survival probab quenched})
\begin{align}\label{eq:survival probability explicit}
\Pi_{n}\,\bfP^{(1:n)}(Z_{n}>0)\ =\ \frac{1}{\bfE^{(1:n)}(Z_{n}|Z_{n}>0)}\ =\ \frac{1}{1+R_{n}/\Pi_{n}}\quad\text{a.s.}
\end{align}
As already indicated, the situation improves under reversal of the environment at each $n$ because this means to replace $R_{n}/\Pi_{n}$ by its a.s.~convergent counterpart $R_{n}^{(-1)}$. We have
\begin{gather}\label{eq:dist Zn given Zn>0 bw}
\bfP^{(n:1)}(Z_{n}\in\cdot|Z_{n}>0)\ =\ \Geom_{+}\left(\frac{1}{1+R_{n}^{(-1)}}\right)
\intertext{and accordingly}
\Pi_{n}\,\bfP^{(n:1)}(Z_{n}>0)\ =\ \frac{1}{\bfE^{(n:1)}(Z_{n}|Z_{n}>0)}\ =\ \frac{1}{1+R_{n}^{(-1)}}\quad\text{a.s.}\label{eq:survival probability explicit bw}
\end{gather}
Using Prop. \ref{prop:duality lemma}, the following quenched convergence result under (C2) is almost immediate.

\begin{Theorem}\label{thm:quenched limit}
Let $(Z_{n})_{n\ge 0}$ be subcritical and $R_{\infty}^{(-1)}<\infty=R_{\infty}$ a.s. Then 
\begin{gather}
\Pi_{n}\,\bfP^{(n:1)}(Z_{n}>0)\ =\ \frac{1}{\bfE^{(n:1)}(Z_{n}|Z_{n}>0)}\ \xrightarrow{n\to\infty}\ \frac{1}{1+R_{\infty}^{(-1)}}\quad\text{a.s.}.\label{eq:survival probability q}
\intertext{and $\bfP^{(n:1)}(Z_{n}\in\cdot|Z_{n}>0)$ converges a.s.~to $\Geom_{+}(1/(1+R_{\infty}^{(-1)}))$ in total variation, that is}
\left\|\bfP^{(n:1)}(Z_{n}\in\cdot|Z_{n}>0)-\Geom_{+}\left(\frac{1}{1+R_{\infty}^{(-1)}}\right)\right\|\ \xrightarrow{n\to\infty}\ 0\quad\text{a.s.}\label{eq:Yaglom q}
\end{gather}
\end{Theorem}

\begin{proof}
As $R_{n}^{(-1)}\to R_{\infty}^{(-1)}$ a.s., the assertions follow directly from \eqref{eq:dist Zn given Zn>0 bw} and \eqref{eq:survival probability explicit bw}.\qed
\end{proof}

Our second quenched result takes a look at the size of the population at the eve of extinction and must therefore condition on $\{Z_{n}>0,\,Z_{n+l}=0\}$ with any fixed $l\in\N$. Put 
$$ h_{n, l}(s)\ :=\ \bfE^{(1:n+l)}\left(s^{Z_{n}}|Z_{n} > 0, Z_{n+l} = 0\right) $$
and recall from \eqref{eq:survival probab quenched} that $q_{n}(\bfe):=f_{1:n}(0)=1-(R_{n}+\Pi_{n})^{-1}$.
Then
\begin{align*}
h_{n, l}(s)\ &=\ \frac{\bfE^{(1:n+l)}\left(s^{Z_{n}}\1_{\{Z_{n} > 0,Z_{n+l} = 0\}}\right)}{\bfP^{(1:n+l)}(Z_{n} > 0,Z_{n+l} = 0)}\\
&=\ \frac{\bfE^{(1:n+l)}\left((sf_{n+1:n+l}(0))^{Z_{n}}\1_{\{Z_{n}>0\}}\right)}{\bfP^{(1:n+l)}(Z_{n} > 0,Z_{n+l} = 0|\bfe)}\\
&=\ \frac{\bfE^{(1:n+l)}\left((sf_{n+1:n+l}(0))^{Z_{n}}\1_{\{Z_{n}>0\}}\right)}{\bfE^{(1:n+l)}\left((f_{n+1:n+l}(0))^{Z_{n}}\1_{\{Z_{n}>0\}}\right)}\\
&=\ \frac{\bfE^{(1:n+l)}\left((sf_{n+1:n+l}(0))^{Z_{n}}|Z_{n}>0\right)}{\bfE^{(1:n+l)}\left((f_{n+1:n+l}(0))^{Z_{n}}|Z_{n}>0\right)}\\
&=\ \frac{h_{n}(f_{n+1:n+l}(0)s)}{h_{n}(f_{n+1:n+l}(0))}\ =\ \frac{h_{n}(q_{l}(\bfe_{n+1:n+l})s)}{h_{n}(q_{l}(\bfe_{n+1:n+l}))},
\end{align*}
and since $h_{n}$ is the g.f.~of the geometric law $\Geom_{+}(\Pi_{n}(1-q_{n}(\bfe_{1:n})))$ (cf.\,\eqref{eq:dist Zn given Zn>0 fw}), we conclude by recalling \eqref{eq:geom law transformation} that $\bfP^{(1:n+l)}(Z_{n}|Z_{n}>0, Z_{n+l}=0)$ is a.s.~geometric as well, the parameter being
\begin{align*}
\vth_{n,l}(\bfe_{1:n+l})\,:=\,1-q_{l}(\bfe_{n+1:n+l})\big(1-\Pi_{n}(1-q_{n}(\bfe_{1:n})\big).
\end{align*}
So this law fluctuates in accordance with $\vth_{n,l}(\bfe_{1:n+l})$ which again converges only in law. Reversal of the environment provides the same result with random parameter
\begin{gather*}
\vth_{n,l}(\bfe_{n+l:1})\ =\ 1-q_{l}(\bfe_{l:1})\left(1-\frac{\Pi_{n+l}}{\Pi_{l}}(1-q_{n}(\bfe_{n+l:l+1})\right)\\
=\ \frac{1}{\Pi_{l}(1+R_{l}^{(-1)})}\,+\,\left(1-\frac{1}{\Pi_{l}(1+R_{l}^{(-1)})}\right)\frac{1}{1+R_{l,n}^{(-1)}}
\end{gather*}
which has obviously the a.s.~limit
$$ \vth_{l}\ :=\ \frac{1}{\Pi_{l}(1+R_{l}^{(-1)})}\,+\,\left(1-\frac{1}{\Pi_{l}(1+R_{l}^{(-1)})}\right)\frac{1}{1+R_{l,\infty}^{(-1)}} $$
as $n\to\infty$, where
$$ R_{l,n}^{(-1)}\ :=\ \sum_{k=1}^{n}\frac{\Pi_{l}}{\Pi_{k+l}}B_{k+l}\ =\ \Pi_{l}\left(R_{n+l}^{(-1)}-R_{l}^{(-1)}\right) $$
for $n\in\N\cup\{\infty\}$ and $R_{l,n}^{(-1)}\eqdist R_{n}^{(-1)}$ for all $l$ and $n$ should be noted. The first two assertions of the subsequent theorem are now immediate.

\begin{Theorem}\label{thm:extinction next stage q}
Let $(Z_{n})_{n\ge 0}$ be subcritical and $l\in\N$. Then
\begin{gather}\label{eq:extinction next stage q}
\left\|\bfP^{(n+l:1)}(Z_{n}\in \cdot|Z_{n}>0,Z_{n+l}=0)-\Geom_{+}(\vth_{l})\right\|\ \xrightarrow{n\to\infty}\ 0\quad\text{a.s.},\\
\bfE^{(n+l:1)}(Z_{n}|Z_{n}>0,Z_{n+l}=0)\ \xrightarrow{n\to\infty}\ \vth_{l}^{-1}\quad\text{a.s.}\label{eq:mean next stage q}
\shortintertext{and}
\begin{split}\label{eq:survival probability next stage q}
&\frac{\Pi_{n+l}}{\Pi_{l}}\,\bfP^{(n+l:1)}(Z_{n}>0,Z_{n+l}=0)\\
&\hspace{2cm}\xrightarrow{n\to\infty}\ \frac{1-(\Pi_{l}+R_{l})^{-1}}{1+R_{l,\infty}^{(-1)}(\Pi_{l}+R_{l})^{-1}}\cdot\frac{1}{1+R_{l,\infty}^{(-1)}}\quad\text{a.s.}
\end{split}
\end{gather}
\end{Theorem}

\begin{proof}
By \eqref{eq:dist Zn given Zn>0 bw},
$$ \bfP^{(n+l:l)}(Z_{n}\in\cdot|Z_{n}>0)\ =\ \Geom_{+}\left(\frac{1}{1+R_{l,n}^{(-1)}}\right)\quad\text{a.s.} $$
and this conditional law is the same under $\bfP^{(n+l:1)}$. Hence
$$ \bfE^{(n+l:1)}\left(s^{Z_{n}}|Z_{n}>0\right)\ =\ \frac{\big(1+R_{l,n}^{(-1)}\big)^{-1}s}{1-\big(1+R_{l,n}^{(-1)}\big)^{-1}R_{l,n}^{(-1)}s}. $$
We further have
\begin{align*}
\frac{\Pi_{n+l}}{\Pi_{l}}\,\bfP^{(n+l:1)}(Z_{n}>0)\ =\ \frac{1}{1+R_{l,n}^{(-1)}}\quad\text{a.s.}
\end{align*}
and so
\begin{align*}
\frac{\Pi_{n+l}}{\Pi_{l}}\,&\bfP^{(n+l:1)}(Z_{n}>0,Z_{n+l}=0)\\
&=\ \frac{\Pi_{n+l}}{\Pi_{l}}\,\bfE^{(n+l:1)}\left(q_{l}(\bfe_{l:1})^{Z_{n}}\1_{\{Z_{n}>0\}}\right)\\
&=\ \bfE^{(n+l:1)}\left(\left(\frac{\Pi_{l}+R_{l}-1}{\Pi_{l}+R_{l}}\right)^{Z_{n}}\bigg|Z_{n}>0\right)\frac{\Pi_{n+l}}{\Pi_{l}}\,\bfP^{(n+l:1)}(Z_{n}>0)\\
&=\ \frac{1-(\Pi_{l}+R_{l})^{-1}}{1+R_{l,n}^{(-1)}(\Pi_{l}+R_{l})^{-1}}\cdot\frac{1}{1+R_{l,n}^{(-1)}}\quad\text{a.s.}
\end{align*}
The asserted limit in \eqref{eq:survival probability next stage q} is now obvious.\qed
\end{proof}

In order to better understand the typical path to extinction, another process of interest in both the subcritical and critical regime is the so-called \emph{reduced GWPRE} $(Z_{m,n})_{0\le m\le n}$ for any $n\in\N$ and under $\bfP^{(1:n)}(\cdot|Z_{n}>0)$. Fixing a time horizon $n$, it accounts for the number of individuals in the population up to that time whose families have not died out by time $n$. More precisely, $Z_{m,n}$ equals the number of individuals in generation $m$ who have descendants in generation $n$. In random environment, it was studied in a series of papers by Vatutin with varying collaborators, often with a focus on annealed limit laws; see for example \cite{BorVat:97,VatDya:97,FleischVa:99,VatDya:02}. The following lemma is not difficult to prove but we refrain from giving the somewhat tedious technical details regarding its second part.

\begin{Lemma}\label{lem:reduced process laws}
The reduced process $(Z_{m,n})_{0\le m\le n}$ is nondecreasing and
\begin{equation}\label{eq:law Z_m,n q fw}
\bfP^{(1:n)}(Z_{m,n}\in\cdot|Z_{n}>0)\ =\ \Geom_{+}\!\left(1-\frac{R_{m}}{\Pi_{n}+R_{n}}\right).
\end{equation}
Moreover, $Z_{m,n}\eqdist\sum_{i=1}^{Z_{l,n}}\zeta_{i}^{l,m}$ for $0\le l<m\le n$, where $\zeta_{1}^{l,m},\zeta_{2}^{l,m},\ldots$ are independent of $Z_{l,n}$ under $\bfP^{(1:n)}(\cdot|Z_{n}>0)$ and further i.i.d.~with common law
\begin{equation}\label{eq:law of zeta_l,k q fw}
\Geom_{+}\!\left(1-\frac{R_{m}-R_{l}}{\Pi_{n}+R_{n}-R_{l}}\right).
\end{equation}
\end{Lemma}

One can interpret $\zeta_{i}^{l,m}$ as the number of individuals in generation $m$ who have descendants in generation $n$ and are stemming from individual $i$ in generation $l$ (who has therefore descendants in generation $n$ as well).

\begin{proof}
As for \eqref{eq:law Z_m,n q fw}, it suffices to note that $Z_{m,n}$ is obtained from $Z_{m}$ by
tagging each individual in generation $m$ which has offspring in generation $n$. Formally, $Z_{m,n}=\sum_{k=1}^{Z_{m}}I_{k}$ where $I_{1},I_{2},\ldots$ are i.i.d.~Bernoulli variables under $\bfP^{(1:n)}$ with parameter
$$ 1-q_{n-m}(\bfe_{m+1:n})\ =\ \left(\frac{\Pi_{n}}{\Pi_{m}}+\frac{R_{n}-R_{m}}{\Pi_{m}}\right)^{-1}\ =\ \frac{\Pi_{m}}{\Pi_{n}+R_{n}-R_{m}}. $$
Since $\bfP^{(1:n)}(Z_{m}\in\cdot)=LF(\Pi_{m},R_{m})$, the law of $Z_{m,n}$ under $\bfP^{(1:n)}$ equals
$$ LF\left(\frac{\Pi_{m}}{1-q_{n-m}(\bfe_{m+1:n})},R_{m}\right)\ =\ LF(\Pi_{n}+R_{n}-R_{m},R_{m}). $$
Now \eqref{eq:law Z_m,n q fw} follows easily when additionally conditioning on the event $\{Z_{n}>0\}$ which is the same as $\{Z_{m,n}>0\}$.\qed
\end{proof}

We also note the conditional law of $Z_{m,n}$ when reversing the environment as in the previous results, thus under $\bfP^{(n:1)}(\cdot|Z_{n}>0)$. As one can readily check,
\begin{gather}\label{eq:law Z_m,n q bw}
\bfP^{(n:1)}(Z_{m,n}\in\cdot|Z_{n}>0)\ =\ \Geom_{+}\!\left(\frac{1+R_{n-m}^{(-1)}}{1+R_{n}^{(-1)}}\right),\\
\bfP^{(n:1)}(\zeta_{1}^{l,m}\in\cdot|Z_{n}>0)\ =\ \Geom_{+}\!\left(\frac{1+R_{n-m}^{(-1)}}{1+R_{n-l}^{(-1)}}\right),\label{eq:law of zeta_l,k q bw}
\end{gather}
and also that the two random parameters appearing in the geometric laws in \eqref{eq:law Z_m,n q fw} and \eqref{eq:law Z_m,n q bw}, and also those in \eqref{eq:law of zeta_l,k q fw} and \eqref{eq:law of zeta_l,k q bw} are identically distributed as they must.

\begin{Theorem}\label{thm:branchless q}
(a) Let $(m_{n})_{n\ge 1}$ be an integer sequence such that $n-m_{n}\to\infty$. Then
\begin{gather}
\lim_{n\to\infty}\Erw\bfP^{(1:n)}(Z_{m_{n},n}>1|Z_{n}>0)\ =\ 0\label{eq:branchless q fw}
\shortintertext{and}
\bfP^{(n:1)}(Z_{m_{n},n}>1|Z_{n}>0)\ \xrightarrow{n\to\infty}\ 0\quad\text{a.s.}\label{eq:branchless q bw}
\intertext{If, furthermore, $\limsup_{n\to\infty}n^{-1}m_{n}<1$}
\label{eq:moment assumption branchless q}
\Erw\log^{2}A\,<\,\infty\quad\text{and}\quad\Erw\log^{2}B\,<\,\infty,
\shortintertext{then also}
\bfP^{(1:n)}(Z_{m_{n},n}>1|Z_{n}>0)\ \xrightarrow{n\to\infty}\ 0\quad\text{a.s.}\label{eq2:branchless q fw}
\end{gather}
holds true.

\vspace{.1cm}
(b) For arbitrary $m\in\N$, let $(Z_{k}^{*})_{0\le k\le m}$ be a GWP in the random environment $\bfe_{m:1}$ such that
$$ \bfP^{(m:1)}(Z_{0}^{*}\in\cdot)\ =\ \Geom_{+}\!\left(\frac{1+R_{m}^{(-1)}}{1+R_{\infty}^{(-1)}}\right). $$
Individuals in generation $k$ are supposed to have the same conditional offspring law as $(Z_{n-m+k,n})_{0\le k\le m}$ under $\bfP^{(n:1)}(\cdot|Z_{n}>0)$ for each $k\in\{0,\ldots,m-1\}$, thus
\begin{align*}
\Geom_{+}\!\left(\frac{1+R_{m-k-1}^{(-1)}}{1+R_{m-k}^{(-1)}}\right).
\end{align*}
Then
\begin{multline}
\left\|\bfP^{(n:1)}((Z_{k,n})_{n-m\le k\le n}\in\cdot|Z_{n}>0)\right.\\
-\left.\bfP^{(m:1)}((Z_{k}^{*})_{0\le k\le m}\in\cdot)\right\|\ \xrightarrow{n\to\infty}\ 0\quad\text{a.s.}
\end{multline}
\end{Theorem}

\begin{proof}
(a) Put $l_{n}:=n-m_{n},\,\rho:=\liminf_{n\to\infty}n^{-1}l_{n}$, and let $\wh{\Pi}_{k}$ be a copy of $\Pi_{k}$ independent of all other occurring random variables. Use \eqref{eq:law Z_m,n q bw}, $l_{n}\to\infty$ and $R_{n}^{(-1)}\to R_{\infty}^{(-1)}$ a.s. to infer
\begin{align*}
\bfP^{(n:1)}(Z_{m_{n},n}>1|Z_{n}>0)\ =\ \frac{R_{n}^{(-1)}-R_{l_{n}}^{(-1)}}{1+R_{n}^{(-1)}}\ \xrightarrow{n\to\infty}\ 0\quad\text{a.s.}
\end{align*}
and thus \eqref{eq:branchless q bw}. But this also implies \eqref{eq:branchless q fw} as
$$ \bfP^{(n:1)}(Z_{m_{n},n}>1|Z_{n}>0)\ \eqdist\ \bfP^{(1:n)}(Z_{m_{n},n}>1|Z_{n}>0) $$
for each $n$ and by the dominated convergence theorem.

\vspace{.1cm}
Turning to \eqref{eq2:branchless q fw} under the stated extra moment conditions, observe that \eqref{eq:law Z_m,n q fw} entails
\begin{align*}
\bfP^{(1:n)}(Z_{m_{n},n}>1|Z_{n}>0)\ &=\ \frac{R_{m_{n}}}{\Pi_{n}+R_{n}}\ \le\ \frac{R_{m_{n}}}{\Pi_{n}+R_{m_{n}}}\ \eqdist\ \frac{R_{m_{n}}^{(-1)}}{\wh{\Pi}_{l_{n}}+R_{m_{n}}^{(-1)}}.
\end{align*}
Assuming \eqref{eq:moment assumption branchless q} and putting $\nu=\Erw\log A$, it follows by the Hsu-Robbins theorem (see \cite[Cor.~10.4.2]{Chow+Teicher:97}) that
\begin{gather*}
\sum_{n\ge 1}\Prob(|S_{n}+n\nu|>\eps n)\ <\ \infty\quad\text{for all }\eps>0,
\shortintertext{and by \cite[Thm.~1.2]{AlsIks:09} that}
\Erw\log(1+R_{\infty}^{(-1)})\ <\ \infty.
\end{gather*}
In order to conclude  \eqref{eq2:branchless q fw}, it suffices to show that
$$ \Prob\Big(\bfP^{(1:n)}(Z_{m_{n},n}>1\Big| Z_{n}>0)>\eps\Big)\ <\ \infty $$
for all $\eps>0$. To this end, observe that
\begin{align*}
\Prob\Big(\bfP^{(1:n)}&(Z_{m_{n},n}>1\Big|Z_{n}>0)>\eps\Big)\ \le\ \Prob\left(\frac{R_{m_{n}}^{(-1)}}{\wh{\Pi}_{l_{n}}+R_{m_{n}}^{(-1)}}>\eps\right)\\
&\ \le\ \Prob\left(\frac{R_{\infty}^{(-1)}}{\wh{\Pi}_{l_{n}}+R_{\infty}^{(-1)}}>\eps\right)\\
&\le\ \Prob\left(R_{\infty}^{(-1)}>e^{\eps n}\right)\,+\,\Prob\left(R_{\infty}^{(-1)}\le e^{\eps n},\frac{e^{\eps n}}{\Pi_{l_{n}}+e^{\eps n}}>\eps\right)\\
&\le\ \Prob\left(\log R_{\infty}^{(-1)}>\eps n\right)\,+\,\Prob\left(S_{l_{n}}>-\eps n-\log(\eps^{-1}-1)\right)\\
&\le\ \Prob\left(\log R_{\infty}^{(-1)}>\eps n\right)\,+\,\Prob\left(S_{l_{n}}+\nu l_{n}>(\nu-2\rho\eps) l_{n}\right)
\end{align*}
the last estimate being valid for sufficiently large $n$ only. But for $\eps$ so small that $\nu-2\rho\eps>0$ (recalling $\nu>0$), we now conclude the summability of the last line over $n\ge 1$ as required.

\vspace{.1cm}
Since, by Lemma \ref{lem:reduced process laws}, $Z_{k+1,n}=\sum_{i=1}^{Z_{k,n}}\zeta_{i}^{k-1,k}$ with $\zeta_{1}^{k-1,k},\zeta_{2}^{k-1,k},\ldots$ as described there, it suffices to note that, by recalling \eqref{eq:law Z_m,n q bw} and \eqref{eq:law of zeta_l,k q bw},
\begin{gather*}
\bfP^{(n:1)}(Z_{n-m,n}\in\cdot|Z_{n}>0)\ \xrightarrow{n\to\infty}\ \Geom_{+}\!\left(\frac{1+R_{m}^{(-1)}}{1+R_{\infty}^{(-1)}}\right)
\intertext{in total variation, and}
\bfP^{(n:1)}(\zeta_{1}^{n-m+k-,n-m+k+1}\in\cdot|Z_{n}>0)\ =\ \Geom_{+}\!\left(\frac{1+R_{m-k-1}^{(-1)}}{1+R_{m-k}^{(-1)}}\right)
\end{gather*}
for each $k=0,\ldots,m-1$.\qed
\end{proof}

Turning to annealed results under the common assumption
\begin{equation}\label{eq:LlogL assumption 1/A}
\Erw f'(1)\log f'(1)\ =\ \Erw(1/A)\log(1/A)\ <\ \infty,
\end{equation}
three subregimes must be distinguished for which the annealed probability of survival $\Prob(Z_{n}>0)$ exhibits a different behavior as $n\to\infty$, see \cite[Sect.\,2.4]{KerstingVatutin:17}, \cite{GeKeVa:03} and also \cite{FleischVa:99} which focusses on the linear fractional case. Let $\psi(\theta):=\log\Erw e^{\theta\log(1/A)}=\log\Erw(1/A)^{\theta}$ denote the cumulant g.f.~of $\log(1/A)$ which, by \eqref{eq:LlogL assumption 1/A}, is finite for $\theta\in [0,1]$ with finite derivative
$$ \psi'(\theta)=e^{-\psi(\theta)}\Erw\log(1/A)(1/A)^{\theta}. $$
A linear fractional GWP in i.i.d.~random environment $(Z_{n})_{n\ge 0}$ is called
\begin{itemize}\itemsep2pt
\item\emph{strongly subcritical} if $\psi'(0)<0$ and $\psi'(1)<0$;
\item\emph{intermediately subcritical} if  $\psi'(0)<0$ and $\psi'(1)=0$;
\item\emph{weakly subcritical} if  $\psi'(0)<0$ and $\psi'(1)>0$.
\end{itemize}
Putting $\kappa:=\psi(1)=\log\Erw(1/A)$, we can define a new measure $\wh{\Prob}$ with the help of the positive martingale $(e^{-\kappa n}\Pi_{n}^{-1})_{n\ge 0}$, namely
$$ \wh{\Prob}(A)\ :=\ \int_{A}\frac{e^{-\kappa n}}{\Pi_{n}}\ d\Prob\ =\ \int_{A}e^{S_{n}-\kappa n}\ d\Prob $$
for $A\in\sigma((A_{k},B_{k})_{0\le k\le n})$ and any $n\in\N_{0}$. Under $\wh{\Prob}$, the $(A_{k},B_{k})$ are again i.i.d.~and $\wh{\Erw}\log(1/A)=e^{\kappa}\,\Erw(1/A)\log(1/A)$. The convexity of $\psi$ together with $\psi'(0)<0$ entails that $\kappa<0$ in the strongly and intermediately subcritical case, in fact
$$ e^{\kappa}\ =\ \Erw(1/A)\ =\ \inf_{0\le\theta\le 1}\Erw(1/A)^{\theta}. $$
As a direct consequence of \eqref{eq:survival probability explicit bw}, we obtain
\begin{align}\label{eq:surv probability a}
\Prob(Z_{n}>0)\ &=\ \Erw\bfP^{(n:1)}(Z_{n}>0)\\
&=\ \Erw\left(\frac{1}{\Pi_{n}(1+R_{n}^{(-1)})}\right)\ =\ e^{\kappa n}\,\wh{\Erw}\left(\frac{1}{1+R_{n}^{(-1)}}\right)
\end{align}
and also the inequality
\begin{align}\label{eq:bounds surv probability a}
\frac{1}{1+\wh\Erw R_{n}^{(-1)}}\ \le\ e^{-\kappa n}\,\Prob(Z_{n}>0)\ \le\ 1
\end{align}
for all $n\in\N$. Here the upper bound is trivial while the lower one follows by Jensen's inequality. We now see that the survival probability essentially behaves like $e^{\kappa n}$ if $\wh\Erw R_{\infty}^{(-1)}$, the monotone limit of $\wh\Erw R_{n}^{(-1)}$, is finite which, as will be seen below, does only hold in the strongly subcritical case.

\vspace{.1cm}
Regarding the annealed conditional law of $Z_{n}$ given $Z_{n}>0$, \eqref{eq:survival probability explicit bw} implies that it is a mixture of geometric laws on $\N$ with a mixing measure on the open unit interval $(0,1)$ that again involves the finite perpetuity $R_{n}^{(-1)}$ in terms of $(1+R_{n}^{(-1)})^{-1}$. More precisely,
\begin{align}
&\Prob(Z_{n}\in\cdot\,| Z_{n}>0)\nonumber\\
&=\ \frac{1}{\Prob(Z_{n}>0)}\int\bfP^{(n:1)}(Z_{n}\in\cdot|Z_{n}>0)\,\bfP^{(n:1)}(Z_{n}>0)\ d\Prob\nonumber\\
&=\ \frac{1}{e^{-\kappa n}\,\Prob(Z_{n}>0)}\int\Geom_{+}\left(\frac{1}{1+R_{n}^{(-1)}}\right)\,\frac{e^{-\kappa n}}{\Pi_{n}(1+R_{n}^{(-1)})}\ d\Prob\nonumber\\
&=\ \int_{(0,1)}\theta\,\Geom_{+}(\theta)\ \Lambda_{n}(d\theta)\label{eq:dist Zn given Zn>0 a}
\end{align}
for each $n\in\N$, where
\begin{equation*}
\Lambda_{n}(d\theta)\ =\ \theta\,\wh{\Prob}\left(\frac{1}{1+R_{n}^{(-1)}}\in d\theta\right)\Bigg\slash\wh{\Erw}\left(\frac{1}{1+R_{n}^{(-1)}}\right)
\end{equation*}
As a particular consequence, using that $\Geom_{+}(\theta)$ has mean $1/\theta$, we find the following annealed analog of \eqref{eq:survival probability explicit}:
\begin{align}\label{eq:surv mean<->surv prob a}
\Erw(Z_{n}|Z_{n}>0)\ =\ \frac{1}{e^{-\kappa n}\,\Prob(Z_{n}>0)}.
\end{align}
In the strongly subcritical case, the previous identities quite directly lead to asymptotic results as $n\to\infty$, summarized in the subsequent theorem.

\begin{Theorem}\label{thm:Yaglom limit}
Let $(Z_{n})_{n\ge 0}$ be strongly subcritical, that is $\Erw(1/A)<\infty$ and $-\infty<\Erw(1/A)\log(1/A)<0$, and also $\Erw\log B<\infty$. Then $\wh{\Prob}(R_{\infty}^{(-1)}<\infty)=1$ and
\begin{equation}\label{eq:Yaglom annealed}
\left\|\Prob(Z_{n}\in\cdot|Z_{n}>0)-\int_{(0,1)}\theta\,\Geom_{+}(\theta)\ \Lambda_{\infty}(d\theta)\right\|\ \xrightarrow{n\to\infty}\ 0,
\end{equation}
with $\Lambda_{\infty}$ defined as $\Lambda_{n}$ above for $R_{\infty}^{(-1)}$. Furthermore,
\begin{gather}
e^{-\kappa n}\,\Prob(Z_{n}>0)\ =\ \frac{1}{\Erw(Z_{n}|Z_{n}>0)}\ \xrightarrow{n\to\infty}\ \wh{\Erw}\left(\frac{1}{1+R_{\infty}^{(-1)}}\right).\label{eq:limit surv probability a}
\end{gather}
\end{Theorem}

\begin{proof}
If $\wh{\Erw}\log(1/A)=e^{-\kappa}\Erw(1/A)\log(1/A)\in\IRl$ and $\Erw\log B<\infty$, then Prop.~\ref{prop:GoldieMaller} ensures that $R_{\infty}^{(-1)}<\infty$ $\wh{\Prob}$-a.s., which in turn implies the weak convergence of $\Lambda_{n}$ to $\Lambda_{\infty}$. Assertion \eqref{eq:Yaglom annealed} is now immediate because the involved distributions are living on $\N$. A combination of \eqref{eq:surv probability a}, \eqref{eq:surv mean<->surv prob a}
and the monotone convergence theorem further provides \eqref{eq:limit surv probability a}.
\qed
\end{proof}

One can also state annealed versions of Theorem \ref{thm:extinction next stage q} and Theorem \ref{thm:branchless q}, for the latter see also \cite[Thm.~2]{FleischVa:99}, but we refrain from doing so here. The neat representation of the asymptotic Yaglom law as a mixture of geometric laws does no longer hold in the intermediately subcritical case, nor in the weakly subcritical case because the a.s.~monotone limit $R_{\infty}^{(-1)}$ of the $R_{n}^{(-1)}$ is no longer finite under the necessary change of measure, at least when ruling out the case that $Ax+B=x$ a.s. for some $x\in\R$. Regarding the intermediately subcritical case, this measure change is the same as in the strongly subcritical case, but since
$$ \wh{\Erw}\log(1/A)\ =\ e^{-\kappa}\,\Erw(1/A)\log(1/A)\ =\ 0, $$
it follows that $\limsup_{n\to\infty}\Pi_{n}^{-1}=\limsup_{n\to\infty}e^{S_{n}}=\infty$ $\wh{\Prob}$-a.s.~by the classical Chung-Fuchs theorem for centered random walks and then $R_{\infty}^{(-1)}=\infty$ $\wh{\Prob}$-a.s.~by another appeal to Prop.~\ref{prop:GoldieMaller}. On the other hand, the existence of an asymptotic Yaglom law can also be shown in these two subregimes (providing some extra conditions), and we refer to \cite{GeKeVa:03} as well as the monograph \cite{KerstingVatutin:17} for more details and an account of further relevant literature.

\section{The supercritical case}\label{sec:supercritical}

Suppose now that $(Z_{n})_{n\ge 0}$ is supercritical, thus $R_{\infty}<\infty=R_{\infty}^{(-1)}$ and also $\Pi_{n}\to 0$ a.s. Recall that $q(\bfe)$ denotes the extinction probability given $\bfe$ and also $\bfe_{\geq n}=(A_{k},B_{k})_{k\ge n}$ for $n\in\N$, thus $\bfe_{\geq 2}=\bfe$. Then it is well-known \cite{AthreyaKarlin:71a} that
\begin{equation}\label{eq:q(bfe)}
q(\bfe_{\geq 1})\ =\ f_{1}(q(\bfe_{\geq 2}))\quad\text{a.s.}
\end{equation}
and that $\{q(\bfe_{\geq 1})=1\}$ is a.s. shift-invariant, i.e.
$$ \{q(\bfe_{\geq 1})=1\}=\{q(\bfe_{\geq 2})=1\}\quad\text{a.s.} $$
Consequently, by ergodicity of the environment,
\begin{equation}\label{eq:0-1 for q(bfe)}
\Prob(q(\bfe_{\geq 1})=1)\ \in\ \{0,1\}.
\end{equation}
In the case when $q(\bfe_{\geq 1})<1$ a.s., we can restate \eqref{eq:q(bfe)} as
\begin{equation}\label{eq:1/1-q(bfe)}
\frac{1}{1-q(\bfe_{\geq 1})}\ =\ \frac{A_{1}}{1-q(\bfe_{\geq 2})}+B_{1},
\end{equation}
thus $q(\bfe_{\geq 1})=\vph^{-1}\circ g_{1}\circ\vph(q(\bfe_{\geq 2}))$. The next theorem is now immediate.

\begin{Theorem}\label{thm:extinction probability}
Let $(Z_{n})_{n\ge 0}$ be supercritical. Then 
\begin{equation}\label{eq:q(bfe) explicit}
\frac{1}{1-q(\bfe)}\ =\ R_{\infty}\ \in\ [1,\infty)\quad\text{a.s.},
\end{equation}
in particular $q(\bfe)<1$ a.s. Furthermore, recalling $\bfP=\Prob(\cdot|\bfe)$,
\begin{equation}\label{eq:surv prob supercritical}
\lim_{n\to\infty}\frac{1}{\Pi_{n}}\left(\frac{1}{R_{n}}-\bfP(Z_{n}>0)\right)\ =\ \frac{1}{R_{\infty}^{2}}\quad\text{a.s.}
\end{equation}
\end{Theorem}

Observe that \eqref{eq:surv prob supercritical} may also be stated as
\begin{equation*}
\bfP(Z_{n}>0)\ =\ \frac{1}{R_{n}}\ -\ \frac{\Pi_{n}}{R_{\infty}^{2}}\ +\ \Gamma_{n}
\end{equation*}
as $n\to\infty$, where $\Gamma_{n}$ satisfies $\Pi_{n}^{-1}\Gamma_{n}\to 0$ a.s.

\begin{proof}
The a.s. finiteness of $R_{\infty}$ follows by Prop. \ref{prop:GoldieMaller}. By iteration of \eqref{eq:1/1-q(bfe)} (or taking the limit in \eqref{eq:survival probab quenched}), we then infer
\begin{align*}
1\ \le\ \frac{1}{1-q(\bfe_{\geq 1})}\ =\ \frac{\Pi_{n}}{1-q(\bfe_{\geq n})}+R_{n}\ \stackrel{n\to\infty}{\longrightarrow}\ R_{\infty}\quad\text{a.s.}
\end{align*}
and therefore $q(\bfe)<1$ a.s. In order to get \eqref{eq:surv prob supercritical}, observe that, by another use of \eqref{eq:survival probab quenched},
\begin{align*}
\frac{1}{\Pi_{n}}\left(\frac{1}{R_{n}}-\bfP(Z_{n}>0)\right)\ =\ \frac{1}{R_{n}(\Pi_{n}+R_{n})}\ \stackrel{n\to\infty}{\longrightarrow}\ \frac{1}{R_{\infty}^{2}}\quad\text{a.s.}
\end{align*}
This completes the proof.\qed
\end{proof}
The following result describes the global limit behaviour of a supercritical process conditioned on non-extinction.
\begin{Theorem}\label{thm:Yaglom limit supercritical}
Let $(Z_{n})_{n\ge 0}$ be supercritical. Then $W_{n}:=\Pi_{n} Z_{n}$ for $n\ge 0$ forms a.s.~a mean one nonnegative martingale under the quenched probability measure $\bfP$ and thus converges a.s.~to a random variable $W_{\infty}$ having conditional law
\begin{equation}\label{eq:Kesten-Stigum}
\bfP(W_{\infty}\in\cdot)\ =\ \left(1-\frac{1}{R_{\infty}}\right)\delta_{0}\ +\ \frac{1}{R_{\infty}}\,\textit{Exp}\left(\frac{1}{R_{\infty}}\right)
\end{equation}
and particularly also mean one.
\end{Theorem}

\begin{proof}
That $(W_{n})_{n\ge 0}$ forms a nonnegative mean one $\bfP$-martingale for almost all realizations of $\bfe$ is a well-known fact. Its a.s.~convergence to some $W_{\infty}$ then follows by the martingale convergence theorem. Moreover, by computing the conditional Laplace transform of $W_{n}$ given $\bfe$, which a.s. converges to the conditional Laplace transform of $W_{\infty}$ given $\bfe$, we find
\begin{align*}
\bfE e^{-uW_{n}}\ &=\ f_{1:n}(e^{-u\Pi_{n}})\ =\ 1-(1-f_{1:n}(0))\frac{1-f_{1:n}(e^{-u\Pi_{n}})}{1-f_{1:n}(0)} \\
&=\ 1-\frac{1}{{R_{n}+\Pi_{n}}}\cdot\frac{R_{n}+\Pi_{n}}{R_{n}-\Pi_{n}/(1-e^{-u\Pi_{n}})}\\
&\hspace{-.68cm}\xrightarrow{n\to\infty}\ 1-\frac{1}{R_{\infty}}\cdot\frac{uR_{\infty}}{1+uR_{\infty}}\\
&=\ \left(1-\frac{1}{R_{\infty}}\right)\ +\ \frac{1}{R_{\infty}}\cdot\frac{R_{\infty}^{-1}}{R_{\infty}^{-1}+u}\quad\text{a.s.}
\end{align*}
and this shows \eqref{eq:Kesten-Stigum}.\qed
\end{proof}

It is well-known, see \cite[p.\,47ff]{Athreya+Ney:72}, that a supercitical GWP $\cZ=(Z_{n})_{n\ge 0}$ with one ancestor and extinction probability $0<q<1$ can be decomposed into two nontrivial parts, say $\cZ_{1}=(Z_{1,n})_{n\ge 0}$ and $\cZ_{2}=(Z_{2,n})_{n\ge 0}$, by dividing each generation into their individuals with a finite line of descent and those with an infinite line of descent, respectively. Then $\cZ_{1}$ and $\cZ_{2}$ are both again nontrivial GWP's, though for the last one the underlying probability measure must be chosen as $\wh\Prob:=\Prob(\cdot|Z_{n}\to\infty)$ (conditioning upon survival). If $f$ denotes the g.f.~of the offspring distribution of $\cZ$, then the offspring distributions of $\cZ_{1}$ and $\cZ_{2}$ have g.f. $g(s)=q^{-1}f(qs)$ and $h(s)=(1-q)^{-1}f(q+(1-q)s)$, respectively. With $P$ denoting the transition kernel of the Markov chain $\cZ$, the law of $\cZ_{1}$, known as the Harris-Sevastyanov transform, is actually nothing but the Doob $h$-transform of $\cZ$ under the positive $P$-harmonic function $\N_{0}\ni i\mapsto q^{i}$, see \cite[Thm.~3.1]{KleRosSag:07}. It is also stated there that, if $f$ is linear fractional, then $g$ is linear fractional, see \cite[Prop.~3.1]{KleRosSag:07}, and it should not be surprising that the same holds true for $h$. More precisely, if the law associated with $f$ is $LF(a,b)$, then the laws associated with $g$ and $h$ are
$$ LF\left(\frac{1}{a},\frac{a+b-1}{a}\right)\quad\text{and}\quad LF(a,1-a)\,=\,\Geom_{+}(a), $$
respectively, where $q=b^{-1}(a+b-1)$ should be recalled.

\vspace{.1cm}
After these preliminary remarks about the fixed environment case, we return to the situation of i.i.d.~random linear fractional offspring laws. Plainly, the decomposition into individuals with finite and infinite line of descent still works and the following result shows that the obtained processes $\cZ_{1},\cZ_{2}$ are again GWPRE's with random  linear fractional reproduction. On the other hand, the environment is no longer i.i.d.

\begin{Theorem}\label{thm:decomposition}
Let $\cZ=(Z_{n})_{n\ge 0}$ be supercritical with $0<q(\bfe)<1$ a.s. Put also $C_{n}:=A_{n}+B_{n}$ {\color{black} and $R_{n,\infty}=\Pi_{n}^{-1}(R_{\infty}-R_{n})$.} Then the following assertions hold in the given notation and for $\cZ_{1}$ and $\cZ_{2}$ as introduced before:
\begin{description}[(b)]\itemsep3pt
\item[(a)] $\cZ_{1}$ is a subcritical GWP in the ergodic stationary environment $(\bfe_{\geq n})_{n\ge 1}$ with quenched random linear fractional offspring law
\begin{gather}{\color{black}
\frac{(C_{n}-1)R_{n,\infty}}{C_{n}(R_{n,\infty}-1)}\,\delta_{0}\ +\ \frac{R_{n,\infty}-C_{n}}{C_{n}(R_{n,\infty}-1)}\,\Geom_{+}\left(\frac{R_{n,\infty}}{C_{n}R_{n+1,\infty}}\right)}\label{eq:laws of Z_1}
\intertext{and associated g.f.}
g_{n}(s)\ =\ \frac{f_{n}(q(\bfe_{\geq n+1})s)}{q(\bfe_{\geq n})}\ =\ \frac{R_{n,\infty}}{R_{n,\infty}-1}f_{n}\left(\frac{(R_{n+1,\infty}-1)s}{R_{n+1,\infty}}\right)
\label{eq:g.f. of Z_1}
\end{gather}
for each $n\ge 1$.
\item[(b)] Conditioned upon $\bfe$ and survival of $\cZ$, i.e.~under $\wh\bfP:=\bfP(\cdot|Z_{n}\to\infty)$, $\cZ_{2}$ is a nonextinctive GWP in varying environment with positive geometric offspring law
\begin{gather}
\Geom_{+}\left(1-\frac{B_{n}}{R_{n,\infty}}\right)\label{eq:laws of Z_2}
\shortintertext{and associated g.f.}
\begin{split}
h_{n}(s)\ &=\ \frac{f_{n}(q(\bfe_{\geq n+1})+(1-q(\bfe_{\geq n+1}))s)-q(\bfe_{\geq n})}{1-q(\bfe_{\geq n})}\\
&=\ \frac{R_{n,\infty}-B_{n})s}{R_{n,\infty}-B_{n}s}
\end{split}
\label{eq:g.f. of Z_2}
\end{gather}
for each $n\ge 1$.
\end{description}
\end{Theorem}

We note that $q(\bfe_{\geq n})$ in the above formulae may also be expressed in terms of  
$\Pi_{n}$ and $R_{n}$ for each $n$. Namely,
$$ R_{\infty}\ =\ \frac{1}{1-q(\bfe_{\geq 1})}\ =\ \frac{\Pi_{n}}{1-q(\bfe_{\geq n})}+R_{n} $$
implies
$$ \frac{1}{1-q(\bfe_{\geq n})}\ =\ \frac{R_{\infty}-R_{n}}{\Pi_{n}}\ {\color{black}=\ R_{n,\infty}\quad\text{and thus}\quad q(\bfe_{\geq n})\ =\ \frac{R_{n,\infty}-1}{R_{n,\infty}}}. $$

\begin{proof}
(a) Consider $\cZ$ under the quenched probability measure $\bfP$. Then an individual $\upsilon$, say, in generation $n-1$ for arbitrary $n\in\N$ produces a random number of offspring with law $LF(A_{n},B_{n})$ and g.f.~$f_{n}$, and each of these children has a finite line of descent with probability $q(\bfe_{\geq n})$, independent {\color{black} of} the other children. Therefore, by exactly the same argument as in the ordinary Galton-Watson case and recalling from \eqref{eq:q(bfe)} that $f_{n}(q(\bfe_{\geq n+1}))=q(\bfe_{\geq n})$, the law of the number of children of $\upsilon$ whose families eventually die out is again linear fractional with the asserted g.f.~$g_{n}$. In order to get \eqref{eq:laws of Z_1}, we argue as follows: It follows by {\color{black}\eqref{eq:LF->Geomplus}} that
$$ LF(A_{n},B_{n})\ = P_{n,0}{\color{black}\,\delta_{0}}\ +\ (1-P_{n,0})\Geom_{+}(P) $$
with ${\color{black}P}=A_{n}/C_{n}$ and ${\color{black}P_{n,0}}=(C_{n}-1)/C_{n}$. Let $\wh{P}_{n,0}$ and $\wh{P}$ denote the corresponding parameters of the law associated with $g_{n}$. Then one can directly see from the relation between $f_{n}$ and $g_{n}$ that $1-\wh{P}=q(\bfe_{\geq n+1})(1-P)$ and
\begin{gather*}
\wh{P}_{n,0}\ =\ g_{n}(0)\ =\ \frac{f_{n}(0)}{q(\bfe_{\geq n})}\ =\ \frac{P_{n,0}}{q(\bfe_{\geq n})}.
\end{gather*}
Finally use \eqref{eq:(a,b)<->(p_0,p)} {\color{black} and again \eqref{eq:LF->Geomplus}} to obtain \eqref{eq:laws of Z_1} after a little algebra.

\vspace{.1cm}
(b) As for \eqref{eq:g.f. of Z_2}, the argument is again the same as in the ordinary Galton-Watson case after conditioning with respect to $\bfe$ and $Z_{n}\to\infty$ and fixing an arbitrary individual $\upsilon$. By regarding its offspring with infinite line of descent, we arrive at a number the (quenched) law of which has indeed the asserted g.f.~$h_{n}$. In contrast to (a), a look at $1/(1-h_{n})$ then also easily provides \eqref{eq:laws of Z_2}. Further details are therefore omitted.\qed
\end{proof}

The following example shows that it is possible to have $q(\bfe_{\geq n})=q$ a.s.~for all $n\in\N_{0}$ and some $q\in (0,1)$ in a truly varying linear fractional environment.

\begin{Exa}\rm
Fix an arbitrary $q\in (0,1)$ and then a sequence $\bfe=(A_{n},B_{n})_{n\ge 1}$ of i.i.d.~nonconstant random vectors with generic copy $(A,B)$ and taking values in $\IRg\times\IRge$ satisfying
\begin{equation*}
\Prob (B>0)<1\quad\text{and}\quad\frac{A+B-1}{B}\,=\,q\quad\text{a.s.}
\end{equation*}
It follows that $A+B=1+Bq\ge 1$ a.s.~and $A=1-B(1-q)\in (0,1]$ a.s.~(Case (C1.2) from Prop.~\ref{prop:trichotomy}), thus our standing assumption \eqref{eq:parameter settings} is fulfilled. Moreover,
\begin{equation*}
\frac{1}{1-q}A+B\ =\ \frac{1}{1-q}\quad\text{a.s.}
\end{equation*}
The last degeneracy property ensures, as it must, that the perpetuity $R_{\infty}$ is a.s.~constant, namely (see \eqref{eq:q(bfe) explicit})
$$ R_{\infty}\ =\ \sum_{n\ge 1}\Pi_{n-1}B_{n}\ =\ \frac{1}{1-q}\sum_{n\ge 1}\Pi_{n-1}(1-A_{n})\ =\ \frac{1}{1-q}\quad\text{a.s.} $$
Writing $LF(A,B)$ as
$$ LF(A,B)\ =\ P_{0}\delta_{0}\ +\ (1-P_{0})\textit{Geom}_{+}(P), $$
we further have
$$ P_{0}\ =\ \frac{Bq}{1+Bq}\quad\text{and}\quad P\ =\ \frac{B}{1+Bq}\quad\text{a.s.} $$
This indicates that the random environment, by means of the random parameter $B\in [0,1/(1-q))$, modulates both the probability for having no offspring $P_{0}$ as well as the tail index $1-P$ of the offspring law, but keeps the extinction probability constant. As a consequence, the g.f.'s $g_{n}$ and $h_{n}$ in \eqref{eq:g.f. of Z_1} and \eqref{eq:g.f. of Z_2}, respectively, of the previous decomposition result also take the much simpler form $g_{n}(s)=q^{-1}f_{n}(qs)$ and $h_{n}(s)=(1-q)^{-1}(f_{n}(q+(1-q)s)-q)$ for each $n\ge 1$ and are thus of the same form as in the ordinary Galton-Watson case. As a further complete analogy, we finally mention that the a.s.~limit $W_{\infty}$ of the normalized martingale $W_{n}=Z_{n}/\Pi_{n}$, $n\ge 0$, has quenched law (see \eqref{eq:Kesten-Stigum})
$$ q\,\delta_{0}\ +\ (1-q)\textit{Exp}(1-q) $$
which does therefore not depend on the environment $\bfe$. On the other hand, the latter enters in the normalization of $Z_{n}$ as shown.
\end{Exa}

\section{The critical case}\label{sec:critical}

Finally, we take a quick look at the critical case when $R_{\infty}=R_{\infty}^{(-1)}=\infty$ and  note first that \eqref{eq:dist Zn given Zn>0 fw} and \eqref{eq:survival probability explicit} are still valid. But unlike the subcritical case, the Markov chain and autoregressive sequence $M_{n}:=R_{n}/\Pi_{n}$, $n\ge 0$, which satisfies the recursion
$$ M_{n}\ =\ \frac{1}{A_{n}}M_{n-1}\ +\ \frac{B_{n}}{A_{n}} $$
and figures in the parameters (see \eqref{eq2:dist Zn given Zn>0 fw} and \eqref{eq2:survival probability explicit} below), is no longer positive recurrent but convergent to $\infty$ in probability $(M_{n}\eqdist R_{n}^{(-1)}\!\uparrow\! R_{\infty}^{(-1)}=\infty$ a.s.$)$. Nonetheless the chain may still exhibit two different kinds of behavior depending on whether it be null recurrent or transient. The following theorem reflects this dichotomy as for its consequences for the quenched survival probability and the quenched conditional law of $Z_{n}$ and its mean given survival. Note that \eqref{eq:survival probability explicit} directly implies that $R_{n}\,\bfP(Z_{n}>0)=M_{n}/\bfE(Z_{n}|Z_{n}>0)$ for each $n\ge 1$.

\begin{Theorem}\label{thm:Yaglom limit critical case}
Let $(Z_{n})_{n\ge 0}$ be critical and thus $R_{\infty}^{(-1)}=\infty=R_{\infty}$ a.s. Denote by $\L=\L(\cdot|\bfe)$ the random set of accumulation points of the sequence $(R_{n}\,\bfP(Z_{n}>0))_{n\ge 1}=(M_{n}/\bfE(Z_{n}|Z_{n}>0))_{n\ge 1}$, and by $\bbD=\bbD(\cdot|\bfe)$ the random set of accumulation points of $(\bfP(Z_{n}\in\cdot|Z_{n}>0))_{n\ge 1}$ with respect to total variation distance. Then
\begin{gather}
R_{n}\,\bfP(Z_{n}>0)\ =\ \frac{M_{n}}{\bfE(Z_{n}|Z_{n}>0)}\ \xrightarrow{n\to\infty}\ 1\quad\text{a.s.},\label{eq:survival probability q}\\
\bfP\left(\frac{Z_{n}}{M_{n}}\in\cdot\bigg|Z_{n}>0\right)\ \weakly\ \textit{Exp}(1)\quad\text{a.s.}\label{eq:quenched exponential law}
\end{gather}
if $(M_{n})_{n\ge 0}$ is transient, whereas
\begin{gather}
\L\ =\ \left\{\frac{x}{1+x}:x\in\R\text{ recurrence point of }(M_{n})_{n\ge 0}\right\}\quad\text{a.s.},\label{eq:limit set survival probability q}\\
\bbD\ =\ \left\{\Geom_{+}\!\left(\frac{1}{1+x}\right):x\in\R\text{ recurrence point of }(M_{n})_{n\ge 0}\right\}\quad\text{a.s.}\label{eq:weak limit set q}
\end{gather}
if $(M_{n})_{n\ge 0}$ is recurrent.
\end{Theorem}

\begin{proof}
By \eqref{eq:dist Zn given Zn>0 fw} and \eqref{eq:survival probability explicit}, we have
\begin{gather}
\bfP(Z_{n}\in\cdot|Z_{n}>0)\ =\ \Geom_{+}\left(\frac{1}{1+M_{n}}\right)\quad\text{a.s.}\label{eq2:dist Zn given Zn>0 fw}
\shortintertext{and}
R_{n}\,\bfP(Z_{n}>0)\ =\ \left(1+\frac{\Pi_{n}}{R_{n}}\right)^{-1}\ =\ \left(1+\frac{1}{M_{n}}\right)^{-1}\quad\text{a.s.}\label{eq2:survival probability explicit}
\end{gather}
Further noting that, if $Y(\theta)\eqdist\Geom_{+}(\theta)$, then $\theta Y(\theta)\idist\textit{Exp}(1)$ as $\theta\downarrow 0$, all assertions are easily verified.\qed
\end{proof}

\begin{Rem}\rm
As for the autoregressive Markov chain $(M_{n})_{n\ge 0}$, it must be acknowledged that, unlike positive recurrence, there seems to be no complete classification of null recurrence and transience of that chain in terms of the random parameter $(A,B)$; for the contractive case when $\Pi_{n}\to 0$ a.s.~we mention the work by the author with Buraczewski and Iksanov \cite{AlsBurIks:17} and by Zerner \cite{Zerner:18}, and for the critical case considered in this section the classical work by Babillot, Bougerol and Elie \cite{BabBouElie:97} and the very recent article by the author with Iksanov \cite{AlsIks:21}. A look at the latter one gives rise to the conjecture that necessary and sufficient conditions are difficult to come by.
\end{Rem}

\begin{Rem}\rm
Since $\bfP^{(n:1)}(Z_{n}\in\cdot|Z_{n}>0)=\Geom_{+}((1+R_{n}^{(-1)})^{-1})$ for any $n$ and $R_{n}^{(-1)}\to\infty$ a.s., the dichotomy encountered in the above theorem disappears under reversal of the environment. Namely,
\begin{equation*}
\bfP^{(n:1)}\left(\frac{Z_{n}}{R_{n}^{(-1)}}\in\cdot\bigg|Z_{n}>0\right)\ \idist\ \textit{Exp}(1)\quad\text{a.s.}
\end{equation*}
\end{Rem}

Let us finally touch very briefly on annealed results by taking the survival probability $\Prob(Z_{n}>0)$ as an example. Eqs.~\eqref{eq2:survival probability explicit} and \eqref{eq:survival probability explicit} provide
$$ \Prob(Z_{n}>0)\ =\ \Erw\left(\frac{1}{R_{n}}\left(1+\frac{1}{M_{n}}\right)^{-1}\right)\ =\ \Erw\left(\frac{1}{\Pi_{n}+R_{n}}\right), $$
and Kozlov \cite{Kozlov:76} embarked on the last expression, written in the form
$$ \Erw\left(e^{-S_{n}}+\sum_{k=1}^{n}e^{-S_{k-1}}B_{k}\right)^{-1} $$
(see his Eq.~(10)), to show that
$$ n^{1/2}\,\Prob(Z_{n}>0)\ \xrightarrow\ \beta $$
for some positive constant $\beta$ under the additional moment conditions
\begin{gather*}
0<\Erw\log^{2}\!A<\infty,\quad\Erw B<\infty\quad\text{and}\quad\Erw B|\log A|<\infty.
\end{gather*}
This result was later extended by Geiger and Kersting \cite{GeiKerst:00} to general critical GWPRE under corresponding moment conditions. Regarding the behavior of the annealed law of $Z_{n}$ given $Z_{n}>0$ in the linear fractional case, we finally mention that Afanasyev \cite{Afanasyev:93,Afanasyev:97} showed the weak convergence of the process $(\Pi_{\lfloor nt\rfloor} Z_{\lfloor nt\rfloor})_{0\le t\le 1}$ in the Skorohod space $D([u,1])$ for each $u\in (0,1)$. In view of Theorem \ref{thm:Yaglom limit critical case}, notably \eqref{eq:quenched exponential law}, this shows that the normings in quenched and annealed regime are different, namely $M_{n}=R_{n}/\Pi_{n}$ versus $1/\Pi_{n}=o(M_{n})$, when the chain $(M_{n})_{n\ge 0}$ is transient. For extensions and further relevant literature, we refer again to \cite[Section 5.8]{KerstingVatutin:17}.

\section{Concluding remarks}

Being aware that our selection of -- essentially known -- results might be seen as somewhat arbitrary and therefore cause reservations of readers especially in the branching process community, we would like to stress once again that we have aimed at offering a different vantage point than earlier publications by adopting a perspective (with respective notation) that is more familiar in the study of random difference equations and their asymptotic properties. By thus putting the focus on the connections of linear fractional GWPRE with these equations, one can observe in a very explicit way how a ``breathing'' or ``fluctuating'' environment impacts on a process evolving in it. Needless to say that we could have discussed many more results, and that there is also plenty of room for extensions, like to the mutlitype setting or even to branching models with interaction. In another direction, Lindo and Sagitov \cite{SagitovLindo:16}, based on the dissertation of the second author \cite{Lindo:16}, have introduced a special class of Galton-Watson processes with explosions, called theta-branching processes, that have similar closure properties as linear fractional branching processes regarding the laws of their marginals. They could therefore be studied in a random environment setting with a similar focus. We refer to future work.

\bibliographystyle{abbrv}
\bibliography{StoPro}

\end{document}